\documentclass{article}
\usepackage{amssymb,amsbsy}
\usepackage{amsmath}
\usepackage{amsthm}
\usepackage{mathrsfs}
\usepackage{color}
\usepackage[dvipsnames]{xcolor}
\usepackage{tikz}

\newtheorem{theorem}{Theorem}[section]
\newtheorem{lemma}[theorem]{Lemma}
\newtheorem{proposition}[theorem]{Proposition}
\newtheorem{corollary}[theorem]{Corollary}

\theoremstyle{definition}

\theoremstyle{remark}
\newtheorem{remark}[theorem]{Remark}
\newtheorem{example}[theorem]{Example}
\newtheorem{algorithm2}[theorem]{Algorithm}

\begin{document}
	
\title{On the enumeration of the set of numerical semigroups with fixed Frobenius number and fixed number of second kind gaps}

\author{Aureliano M. Robles-P\'erez\thanks{Both authors are supported by the project MTM2017-84890-P (funded by Mi\-nis\-terio de Econom\'{\i}a, Industria y Competitividad and Fondo Europeo de Desarrollo Regional FEDER) and by the Junta de Andaluc\'{\i}a Grant Number FQM-343.} \thanks{Departamento de Matem\'atica Aplicada \& Instituto de Matem\'aticas (IMAG), Universidad de Granada, 18071-Granada, Spain. \newline E-mail: {\bf arobles@ugr.es} (\textit{corresponding author}); ORCID: \textbf{0000-0003-2596-1249}.}
	\mbox{ and} Jos\'e Carlos Rosales$^*$\thanks{Departamento de \'Algebra \& Instituto de Matem\'aticas (IMAG), Universidad de Granada, 18071-Granada, Spain. \newline E-mail: {\bf jrosales@ugr.es}; ORCID: \textbf{0000-0003-3353-4335}.} }

\date{ }

\maketitle

\begin{abstract} 
	We study how certain invariants of numerical semigroups relate to the number of second kind gaps. Furthermore, given two fixed non-negative integers $F$ and $k$, we provide an algorithm to compute all the numerical semigroups whose Frobenius number is $F$ and which have exactly $k$ second kind gaps.
\end{abstract}
\noindent {\bf Keywords:} Irreducible numerical semigroups, second kind gaps, Frobenius number, trees, Wilf's conjecture.

\medskip

\noindent{\it 2010 AMS Classification:} 20M14, 11Y16.	

\section{Introduction}

If $a_1,\ldots,a_e$ are positive integers such that $\gcd(a_1,\ldots,a_e)=1$, then a classical problem in additive number theory is the Frobenius problem: what is the greatest integer $F$ which is not an element of the set $a_1{\mathbb N}+\ldots+a_e{\mathbb N}$? Although this problem is solved when $e=2$ (see \cite{sylvester}), it is well known that it is not possible to find a polynomial formula in order to compute $F$ if $e\geq3$ (see \cite{curtis}). Therefore, many efforts have been made to obtain partial results or to develop algorithms to get the answer of this question (for instance, see \cite{alfonsin}).

Before continuing, let us recall some definitions and notations used in numerical semigroups.

Let ${\mathbb Z}$ and ${\mathbb N}$ be the set of integers and the set of non-negative integers respectively. A \emph{numerical semigroup} is a subset $S$ of ${\mathbb N}$ which is closed under addition, $0 \in S$, and ${\mathbb N} \setminus S$ is finite.

The elements of ${\mathbb N}\setminus S$ are called the \emph{gaps} of $S$, and its cardinality, denoted by $\mathrm{g}(S)$, is the \emph{genus} of $S$.

The \emph{Frobenius number} of $S$, denoted by $\mathrm{F}(S)$, is the greatest integer that does not belong to $S$.

The \emph{conductor} of $S$, denoted by $\mathrm{c}(S)$, is the least integer $c$ such that $c+n\in S$ for all $n\in\mathbb{N}$. Note that $\mathrm{c}(S)\in S$ and $\mathrm{c}(S) = \mathrm{F}(S)+1$.

The \emph{pseudo-Frobenius numbers} of $S$ are the elements of the set $\mathrm{PF}(S)= \big\{x\in {\mathbb Z} \setminus S \mid x+s\in S \mbox{ for all } s\in S \setminus \{0\} \big\}$ (see \cite{JPAA}). Moreover, the cardinality of $\mathrm{PF}(S)$, denoted by $\mathrm{t}(S)$, is the \emph{type} of $S$ (see \cite{froberg}).

We denote by $\mathrm{N}(S)=\{s\in S \mid s<\mathrm{F}(S) \}$ (whose elements are known as \emph{small elements} of $S$). It is clear that the sets $\mathrm{H}(S)=\{\mathrm{F}(S)-s\mid s\in\mathrm{N}(S) \}$ and $\mathrm{L}(S)=\{x \in {\mathbb N}\setminus S \mid \mathrm{F}(S)-x\in{\mathbb N}\setminus S\}$ (equivalently, $\mathrm{L}(S)=\{x \in {\mathbb N}\setminus S \mid \mathrm{F}(S)-x\notin\mathrm{N}(S) \}$) are subsets of ${\mathbb N}\setminus S$ and, in addition, ${\mathbb N}= S \mathbin{\mathaccent\cdot\cup} \mathrm{H}(S) \mathbin{\mathaccent\cdot\cup} \mathrm{L}(S)$ (that is, $S$, $H(S)$, and $L(S)$ define a partition of $\mathbb N$). Following the notation in \cite{jager} (see also \cite{barucci-froberg}), the elements of $\mathrm{H}(S)$ and $\mathrm{L}(S)$ are the \emph{first and second kind gaps} of $S$, respectively. Moreover, we denote by $\mathrm{n}(S)$ and $\mathrm{l}(S)$ the cardinality of $\mathrm{N}(S)$ (or $\mathrm{H}(S)$) and $\mathrm{L}(S)$, respectively.

Now, let $A$ be a non-empty subset of ${\mathbb N}$. Then we denote by $\langle A \rangle$ the submonoid of $({\mathbb N},+)$ generated by $A$, that is,
	\[\langle A \rangle=\big\{\lambda_1a_1+\cdots+\lambda_na_n \mid n\in{\mathbb N}\setminus \{0\}, \ a_1,\ldots,a_n\in A, \ \lambda_1,\ldots,\lambda_n\in {\mathbb N}\big\}.\]
It is well known that $\langle A \rangle$ is a numerical semigroup if and only if $\gcd(A)=1$. On the other hand, if $S$ is a numerical semigroup and $S=\langle A \rangle$, then we say that $A$ is a \emph{system of generators} of $S$. In addition, if $S\not=\langle B \rangle$ for any subset $B\subsetneq A$, then we say that $A$ is a \emph{minimal system of generators} of $S$. In \cite{springer} it is shown that each numerical semigroup has a unique minimal system of generators and that such a system is finite. We denote by $\mathrm{msg}(S)$ the minimal system of generators of $S$ and by $\mathrm{e}(S)$ the cardinality of $\mathrm{msg}(S)$, that is called the \emph{embedding dimension} of $S$.

Let $l\in {\mathbb N}$. We say that a numerical semigroup $S$ is an $l$-semigroup if $\mathrm{l}(S)=l$. Our main purpose in this work is to give an algorithm which enables us to build all the $l$-semigroups with a fixed Frobenius number.

Let us summarize the content of this work. In Section~\ref{sect-01-ns} we show that the concepts of $0$-semigroup and $1$-semigroup coincide with the concepts of symmetric numerical semigroup and pseudo-symmetric numerical semigroup, respectively. These two classes of semigroups are of interest and have been widely studied (for instance, see \cite{kunz} and \cite{barucci}). Moreover, as it is shown in \cite{pacific}, they define a partition of the family of irreducible numerical semigroups.

In Section~\ref{sect-23-ns} we relate the $2$-semigroups ($3$-semigroups respectively) with the URSY-semigroups (URPSY-semigroups respectively), which were studied in \cite{colloquium} (\cite{bolletino} respectively). Moreover, as a consequence, we are able to obtain the pseudo-Frobenius numbers of the $2$-semigroups and $3$-semigroups.

In Section~\ref{sect-general-ns} we prove that, if $l$ is an even number (odd number, respectively), then any $l$-semigroup with Frobenius number $F$ can be obtained from a symmetric (pseudo-symmetric, respectively) numerical semigroup with Frobenius number $F$ after removing $\frac{l}{2}$ ($\frac{l-1}{2}$, respectively) elements greater than $\frac{F}{2}$ and less than $F$.

At last, in Section~\ref{sect-k-ns}, we give an algorithm that enables us to compute all the $l$-semigroups with a fixed Frobenius number.

We end this introduction with a brief comment on Wilf's conjecture. Showing an algorithm to solve the Frobenius problem (among other purposes), in \cite{wilf}, Wilf conjectured that, if $S$ is a numerical semigroup, then $\frac{\mathrm{g}(S)}{\mathrm{c}(S)}\leq1-\frac{1}{\mathrm{e}(S)}$. At present, this conjecture is open and its resolution is one of the most important problems in numerical semigroup theory (for instance, see the survey \cite{delgado-2}). In Section~\ref{sect-wilf} we show the above inequality in terms of the second king gaps and recover some known results about it.

\section{0-semigroups and 1-semigroups}\label{sect-01-ns}

A numerical semigroup $S$ is \emph{irreducible} if it cannot be expressed as the intersection of two numerical semigroups which contain $S$ properly. This concept was introduced in \cite{pacific}, where it is shown that a numerical semigroup $S$ is irreducible if and only if $S$ is maximal (with respect to the set inclusion) in the set of numerical semigroups with Frobenius number equal to $\mathrm{F}(S)$. From the results in \cite{barucci} and \cite{froberg}, it is deduced in \cite{pacific} that the family of irreducible numerical semigroups is the union of two well known families, namely the symmetric and the pseudo-symmetric numerical semigroups. Moreover, a numerical semigroup is symmetric (pseudo-symmetric, respectively) if it is irreducible and its Frobenius number is odd (even, respectively).

The following result is Corollary~4.5 of \cite{springer}.

\begin{lemma}\label{lem1}
	Let $S$ be a numerical semigroup.
	\begin{enumerate}
		\item $S$ is symmetric if and only if $\mathrm{g}(S)=\frac{\mathrm{F}(S)+1}{2}$.
		\item $S$ is pseudo-symmetric if and only if $\mathrm{g}(S)=\frac{\mathrm{F}(S)+2}{2}$.
	\end{enumerate}
\end{lemma}

The next result has an easy proof and, therefore, we omit it.

\begin{lemma}\label{lem2}
	If $S$ is a numerical semigroup, then
	\begin{enumerate}
		\item $\mathrm{g}(S)+\mathrm{n}(S)=\mathrm{F}(S)+1$;
		\item $\mathrm{g}(S)=\mathrm{n}(S)+\mathrm{l}(S)$;
		\item $\mathrm{g}(S)=\frac{\mathrm{F}(S)+1+\mathrm{l}(S)}{2}$.
	\end{enumerate}
\end{lemma}

As a direct consequence of Lemmas~\ref{lem1} and \ref{lem2}, we have the following result.

\begin{proposition}\label{prop3}
	If $S$ is a numerical semigroup, then
	\begin{enumerate}
		\item $S$ is symmetric if and only if $\,\mathrm{l}(S)=0$;
		\item $S$ is pseudo-symmetric if and only if $\,\mathrm{l}(S)=1$.
	\end{enumerate}
\end{proposition}

From Lemma~\ref{lem1} we also deduce a useful result that will be used several times in Section~\ref{sect-wilf}.

\begin{corollary}\label{cor_odd}
	If $S$ is a numerical semigroup different from $\mathbb{N}$, then $\mathrm{F}(S)+\mathrm{l}(S)$ is an odd number and $\mathrm{F}(S)\geq\mathrm{l}(S)+1$.
\end{corollary}

\begin{proof}
	From item 3 of Lemma~\ref{lem1}, it is clear that $\mathrm{F}(S)+\mathrm{l}(S)$ is an odd number (even when $S=\mathbb{N}$). On the other hand, if $S$ is a numerical semigroup different from $\mathbb{N}$, then $0\in\mathrm{N}(S)$ and, consequently, $\mathrm{n}(S)\geq 1$. Now, combining items 1 and 2 of Lemma~\ref{lem1}, we have that $\mathrm{F}(S)\geq\mathrm{l}(S)+1$.
\end{proof}

The next result is consequence of Corollaries~4.11 and 4.16 of \cite{springer}. 

\begin{lemma}\label{lem4}
	Let $S$ be a numerical semigroup.
	\begin{itemize}
		\item[a)] The following conditions are equivalent.
		\begin{itemize}
			\item[a.1)] $S$ is symmetric.
			\item[a.2)] $\mathrm{PF}(S)=\{\mathrm{F}(S) \}$.
			\item[a.3)] $\mathrm{t}(S)=1$.
		\end{itemize}
		\item[b)] $S$ is pseudo-symmetric if and only if $\mathrm{PF}(S)=\left\{\mathrm{F}(S),\frac{\mathrm{F}(S)}{2} \right\}$.
	\end{itemize}
\end{lemma}

As an immediate consequence of Proposition~\ref{prop3} and Lemma~\ref{lem4}, we have the next result.

\begin{corollary}
	Let $S$ be a numerical semigroup.
	\begin{enumerate}
		\item $\mathrm{l}(S)=0$ if and only if $\mathrm{PF}(S)=\{\mathrm{F}(S) \}$.
		\item $\mathrm{l}(S)=1$ if and only if $\mathrm{PF}(S)=\left\{\mathrm{F}(S),\frac{\mathrm{F}(S)}{2} \right\}$.
	\end{enumerate}
\end{corollary}

\begin{remark}\label{rem-asns}
	Let us recall that a numerical semigroup $S$ is \textit{almost symmetric} if $\mathrm{L}(S) \subseteq \mathrm{PF}(S)$. This class of numerical semigroups was introduced in \cite{barucci-froberg} and it is characterized by the formula
	\[ \mathrm{g}(S)=\frac{\mathrm{F}(S)+\mathrm{t}(S)}{2} \]
	and by the equality 
	\[ \mathrm{PF}(S)=\mathrm{L}(S)\cup \{\mathrm{F}(S)\} .\]
	Therefore, the set of almost symmetric numerical semigroups is exactly the set of $(t-1)$-semigroups with type $t$. In particular, and as it is well known, every irreducible numerical semigroup is almost symmetric.
\end{remark}

\section{2-semigroups and 3-semigroups}\label{sect-23-ns}

We begin this section studying the numerical semigroups $S$ such that $\mathrm{l}(S)=2$. For that we need to introduce several concepts and results.

If $F$ is a positive integer, then we denote by $\mathscr{S}(F)$ the set of all numerical semigroups with Frobenius number $F$.

The following result is deduced from \cite{froberg} (or from Lemmas~4 and 5 in \cite{computer}).

\begin{lemma}\label{lem6}
	Let $S$ be a numerical semigroup with Frobenius number $F$.
	\begin{enumerate}
		\item $S$ is irreducible if and only if $S$ is maximal in $\mathscr{S}(F)$.
		\item If $\,h=\max\left\{x\in {\mathbb N} \setminus S \,\big\vert\, F-x \in {\mathbb N} \setminus S \mbox{ and } x\not=\frac{F}{2} \right\}$, then $S\cup\{h\} \in \mathscr{S}(F)$.
		\item $S$ is irreducible if and only if  $\left\{x\in {\mathbb N} \setminus S \,\big\vert\, F-x \in {\mathbb N} \setminus S \mbox{ and } x\not=\frac{F}{2} \right\} = \emptyset$.
		\item $S \cup \left\{x\in {\mathbb N} \setminus S \,\big\vert\, F-x \in {\mathbb N} \setminus S \mbox{ and } \, x>\frac{F}{2} \right\}$ is a numerical semigroup that is irreducible and with Frobenius number $F$.
	\end{enumerate}
\end{lemma} 

The next result has an immediate proof.

\begin{lemma}\label{lem7}
	Let $S$ be a numerical semigroup and $x\in S$. Then $S\setminus \{x\}$ is a numerical semigroup if and only if $x\in\mathrm{msg}(S)$.
\end{lemma}

Now, from Lemma~\ref{lem2} we deduce the following one.

\begin{lemma}\label{lem8}
	Let $S$ be a numerical semigroup and $x\in\mathrm{msg}(S)$ such that $x< \mathrm{F}(S)$. Then $\mathrm{l}(S\setminus \{x\})=\mathrm{l}(S)+2$.
\end{lemma}

Let $S$ be a numerical semigroup with $\mathrm{l}(S)\geq 2$. Then, from Proposition~\ref{prop3}, we know that $S$ is not irreducible. Moreover, from Lemma~\ref{lem6}, there exists $\mathrm{h}(S)=\max\left\{x\in {\mathbb N} \setminus S \,\big\vert\, \mathrm{F}(S)-x \in {\mathbb N} \setminus S \mbox{ and } \, x\not=\frac{\mathrm{F}(S)}{2} \right\} = \max\big(\mathrm{L}(S)\big)$ and $S\cup \{\mathrm{h}(S)\}$ is a numerical semigroup. In addition, it is clear that $\mathrm{h}(S)$ is a minimal generator of $S\cup \{\mathrm{h}(S)\}$ which is less than $\mathrm{F}(S\cup \{\mathrm{h}(S)\})$. Therefore, by applying Lemma~\ref{lem8}, we have the next result.

\begin{lemma}\label{lem9}
	If $S\in \mathscr{S}(F)$ and $\,\mathrm{l}(S)\geq2$, then $\mathrm{l}(S\cup \{\mathrm{h}(S)\})=\mathrm{l}(S)-2$.
\end{lemma}

In the following result we show how to obtain all the $l$-semigroups when we know the $(l-2)$-semigroups set.

\begin{proposition}\label{prop10}
	If $S\in \mathscr{S}(F)$, then the following conditions are equivalent.
	\begin{enumerate}
		\item \label{prop10-1} $\mathrm{l}(S)=l$ and $\,l\geq2$.
		\item \label{prop10-2} There exist $\,T\in \mathscr{S}(F)$ and $x\in \mathrm{msg}(T)$ such that $\,\mathrm{l}(T)=l-2$, $\frac{F}{2}<x<F$, and $S=T\setminus \{x\}$.
	\end{enumerate}
\end{proposition}

\begin{proof}
	\textit{(\ref{prop10-1}.} $\Rightarrow$ \textit{\ref{prop10-2}.)} Since $l\geq2$, then there exists $\mathrm{h}(S)$. Therefore, $T=S\cup \{\mathrm{h}(S)\} \in \mathscr{S}(F)$ and $\mathrm{h}(S)$ is a minimal generator of $T$ such that $\frac{F}{2}<\mathrm{h}(S)<F$. Moreover, from Lemma~\ref{lem9}, $\mathrm{l}(T)=\mathrm{l}(S)-2$. In order to finish the proof, it is enough to observe that $S=T\setminus \{\mathrm{h}(S)\}$.
	
	\textit{(\ref{prop10-2}.} $\Rightarrow$ \textit{\ref{prop10-1}.)} This implication is an immediate consequence of Lemma~\ref{lem8}.
\end{proof}

Let us observe that, as a consequence of Lemma~\ref{lem8} and Propositions~\ref{prop3} and \ref{prop10}, if $Q$ is a numerical semigroup with $\mathrm{l}(Q)=2$, then there exist a symmetric numerical semigroup $P$ and a number $x\in \mathrm{msg}(P)$ such that $x<\mathrm{F}(P)$ and $Q=P\setminus \{x\}$. Now, using the study about URSY-semigroups made in \cite{colloquium}, we obtain information about the pseudo-Frobenius numbers of a $2$-semigroup.

Recalling that a numerical semigroup $S$ is an \emph{URSY-semigroup} if there exists a symmetric numerical semigroup $T$ and an $x\in\mathrm{msg}(T)$ such that $S=T\setminus \{x\}$, then we can assert that every $2$-semigroup is an URSY-semigroup. However, the converse is not true. For example, $\langle 2,5 \rangle \setminus \{5\}=\langle 2,7 \rangle$ is an URSY-semigroup and $\mathrm{l}(\langle 2,7 \rangle)=0$.

The next result can be easily deduced from Theorem~2.2 of \cite{colloquium}. Recall that, if $S$ is a numerical semigroup, then $\mathrm{m}(S)=\min(S\setminus \{0\})$ is known as the \emph{multiplicity} of $S$.

\begin{proposition}\label{prop11}
	Let $S$ be a numerical semigroup. Then $\mathrm{l}(S)=2$ if and only if $S$ is an URSY-semigroup and $\mathrm{m}(S)\geq 3$.
\end{proposition}

The following result is obtained from Lemmas~2.3 and 2.7 of \cite{colloquium} and the proof of Proposition~\ref{prop10}.

\begin{proposition}\label{prop12}
	If $S$ is a $2$-semigroup, then
	\[\left\{\mathrm{F}(S),\mathrm{h}(S) \right\} \subseteq \mathrm{PF}(S) \subseteq \{\mathrm{F}(S),\mathrm{h}(S),\mathrm{F}(S)-\mathrm{h}(S) \}.\]
	Moreover, $\mathrm{F}(S)-\mathrm{h}(S)\in\mathrm{PF}(S)$ if and only if $\,2\mathrm{h}(S)-\mathrm{F}(S)\notin S$.
\end{proposition}

An immediate consequence of the above proposition is the next result.

\begin{corollary}\label{cor13}
	If $S$ is a $2$-semigroup, then $\mathrm{t}(S)\in\{2,3\}$. Moreover, $\mathrm{t}(S)=2$ if and only if $\,2\mathrm{h}(S)-\mathrm{F}(S)\in S$.
\end{corollary}

\begin{remark}
	From Remark~\ref{rem-asns} and Proposition~\ref{prop12}, we have that a $2$-semigroup is almost symmetric if and only if $\,2\mathrm{h}(S)-\mathrm{F}(S)\notin S$.
\end{remark}

Let us illustrate the previous results with an example.

\begin{example}\label{exmp14}
	From Lemma~\ref{lem1}, we easily deduce that $S=\langle 7,8,9,10,11,12 \rangle$ is a symmetric numerical semigroup with Frobenius number $13$. Therefore, by applying Proposition~\ref{prop10}, we have that $\mathrm{l}(S\setminus \{10\})=2$ and that $\mathrm{h}(S\setminus \{10\})=10$. Then, since $2\mathrm{h}(S\setminus \{10\})-\mathrm{F}(S\setminus \{10\})=7 \in S\setminus \{10\}$, from Proposition~\ref{prop12} (or Corollary~\ref{cor13}) we conclude that $\mathrm{PF}(S\setminus \{10\})=\{10,13\}$.
\end{example}

Now our purpose is to study the numerical semigroups $S$ such that $\mathrm{l}(S)=3$. As a consequence of Lemma~\ref{lem8} and Propositions~\ref{prop3} and \ref{prop10}, we have that, if $Q$ is a $3$-semigroup, then there exists a pseudo-symmetric numerical semigroup $P$ and an $x\in\mathrm{msg}(P)$ such that $\mathrm{F}(P)=\mathrm{F}(Q)$, $x<\mathrm{F}(P)$, and $Q=P\setminus\{x\}$.

We use the study made in \cite{bolletino} in order to obtain results about the pseudo-Frobenius numbers of the $3$-semigroups.

Let us recall that a numerical semigroup $S$ is an URPSY-semigroup if there exists a pseudo-symmetric numerical semigroup $T$ and an $x\in\mathrm{msg}(T)$ such that $S=T\setminus \{x\}$.

Let us observe that, if $S$ is a $3$-semigroup, then $S$ is an URPSY-semigroup. However, $\langle 3,4,5 \rangle \setminus \{4\}=\langle 3,5,7 \rangle$ is an URPSY-semigroup and $\mathrm{l}(\langle 3,5,7 \rangle)=1$. Consequently, there exist URPSY-semigroups that are not $3$-semigroups.

The following result can be deduced from Proposition~23 of \cite{bolletino}.

\begin{proposition}\label{prop15}
	Let $S$ be a numerical semigroup. Then $\mathrm{l}(S)=3$ if and only if $S$ is an URPSY-semigroup and does not belong to the set
	\begin{align*}
		U=\big\{\langle 3,5,7 \rangle,\langle 4,5,6,7 \rangle,\langle 4,5,11 \rangle \big\} \cup \big\{ \langle 3,x+3 \rangle \mid x\in{\mathbb N}\setminus \{0\}, \ 3\nmid x \big\} \, \cup \\
		\big\{ \langle m,m+1,\ldots,2m-3 \rangle \mid m \in {\mathbb N}, \ m\geq5 \big\}.
	\end{align*}
\end{proposition}

The next result is deduced from Lemmas~25, 26, and 27 of \cite{bolletino}.

\begin{proposition}\label{prop16}
	If $S$ is a $3$-semigroup, then 
	\begin{enumerate}
		\item $\left\{\mathrm{F}(S),\mathrm{h}(S) \right\} \subseteq \mathrm{PF}(S) \subseteq \left\{\mathrm{F}(S),\mathrm{h}(S),\frac{\mathrm{F}(S)}{2},\mathrm{F}(S)-\mathrm{h}(S) \right\};$
		\item $\frac{\mathrm{F}(S)}{2} \in \mathrm{PF}(S)$ if and only if $\,\mathrm{h}(S)-\frac{\mathrm{F}(S)}{2} \notin S$;
		\item $\mathrm{F}(S)-\mathrm{h}(S)\in\mathrm{PF}(S)$ if and only if $\,2\mathrm{h}(S)-\mathrm{F}(S)\notin S$.
	\end{enumerate}
\end{proposition}

As a direct consequence of the above proposition, we have the following result.

\begin{corollary}\label{cor17}
	If $S$ is a $3$-semigroup, then $\mathrm{t}(S) \in \{2,3,4\}$. Moreover,
	\begin{itemize}
		\item[a)] the following conditions are equivalent.
		\begin{itemize}
			\item[a.1)] $\mathrm{t}(S)=2$.
			\item[a.2)] $\mathrm{h}(S)-\frac{\mathrm{F}(S)}{2} \in S$.
			\item[a.3)] $\mathrm{PF}(S)=\left\{\mathrm{F}(S),\mathrm{h}(S) \right\}$.
		\end{itemize}
		\item[b)] the following conditions are equivalent.
		\begin{itemize}
			\item[b.1)] $\mathrm{t}(S)=3$.
			\item[b.2)] $\mathrm{h}(S)-\frac{\mathrm{F}(S)}{2} \notin S$ and $\,2\mathrm{h}(S)-\mathrm{F}(S)\in S$.
			\item[b.3)] $\mathrm{PF}(S)=\left\{\mathrm{F}(S),\mathrm{h}(S),\frac{\mathrm{F}(S)}{2} \right\}$.
		\end{itemize}
		\item[c)] the following conditions are equivalent.
		\begin{itemize}
			\item[c.1)] $\mathrm{t}(S)=4$.
			\item[c.2)] $2\mathrm{h}(S)-\mathrm{F}(S)\notin S$.
			\item[c.3)] $\mathrm{PF}(S)=\left\{\mathrm{F}(S),\mathrm{h}(S),\frac{\mathrm{F}(S)}{2},\mathrm{F}(S)-\mathrm{h}(S) \right\}$.
		\end{itemize}
	\end{itemize}
\end{corollary}

\begin{remark}
	Having in mind Remark~\ref{rem-asns}, it is clear that, in the above corollary, the case $c$ corresponds to the $3$-numerical semigroups which are almost symmetric.
\end{remark}

In the next example we illustrate the above results.

\begin{example}\label{exmp18}
	From Lemma~\ref{lem1}, $S=\langle 8,9,10,11,12,13,15 \rangle$ is a pseudo-symmetric numerical semigroup with Frobenius number $14$. Therefore, by applying Proposition~\ref{prop10}, we have that $\mathrm{l}(S\setminus \{10\})=3$ and that $\mathrm{h}(S\setminus \{10\})=10$. Since $2\mathrm{h}(S\setminus \{10\})-\mathrm{F}(S\setminus \{10\})=6 \notin S\setminus \{10\}$, from item \textit{c} of Corollary~\ref{cor17} we conclude that $\mathrm{PF}(S\setminus \{10\})=\{14,10,7,4\}$.
\end{example}

\section{$\boldsymbol{2n}$-semigroups and $\boldsymbol{(2n+1)}$-semigroups}\label{sect-general-ns}

Let $n\in {\mathbb N}$. We begin this section showing how we can obtain $2n$-semigroups and $(2n+1)$-semigroups from symmetric and pseudo-symmetric numerical semigroups, respectively. (Let us recall that, if $A$ is a set, then $\#(A)$ is the cardinality of $A$.)

\begin{proposition}\label{prop19}
	Let $S$ be a symmetric numerical semigroup. Let us take a set $A \subseteq \left\{x\in S \,\big\vert\, \frac{\mathrm{F}(S)}{2}<x<\mathrm{F}(S) \right\}$ such that $\#(A)=n$ and $S\setminus A$ is a numerical semigroup. Then $\mathrm{l}(S\setminus A)=2n$. Moreover, each $2n$-semigroup can be obtained in this way.
\end{proposition}

\begin{proof}
	It is clear that, if $x\in A$, then $x\notin S\setminus A$ and $\mathrm{F}(S\setminus A)-x\notin S\setminus A$ (note that $\mathrm{F}(S\setminus A)=\mathrm{F}(S)$). Therefore, $\mathrm{l}(S\setminus A)\geq 2n$. On the other hand, having in mind that $S$ is symmetric, we deduce that, if $\left\{h,\mathrm{F}(S\setminus A)-h\right\}\cap (S\setminus A)=\emptyset$, then $\max\left\{h,\mathrm{F}(S\setminus A)-h \right\}$ $\in A$. Thus, $\mathrm{l}(S\setminus A)\leq 2\#(A)=2n$. Consequently, $\mathrm{l}(S\setminus A)=2n$.
	
	Now, let $T$ be a numerical semigroup with $\mathrm{l}(T)=2n$. We define the following sequence of numerical semigroups (recalling item 2 of Lemma~\ref{lem6}).
	\begin{itemize}
		\item $T_0=T$;
		\item $T_{k+1}=T_k\cup \{\mathrm{h}(T_k)\}$, if $0\leq k \leq n-1$.
	\end{itemize}
	Then, by applying Lemma~\ref{lem9}, we have that $T=T_0\subsetneq T_1 \subsetneq \cdots \subsetneq T_n$ and $\mathrm{l}(T_n)=0$. Moreover, from Proposition~\ref{prop3}, $T_n$ is a symmetric numerical semigroup. Finally, $A=T_n \setminus T \subset \left\{x\in T_n \,\big\vert\, \frac{\mathrm{F}(T_n)}{2}<x<\mathrm{F}(T_n) \right\}$ and $\#(A)=n$.
\end{proof}

Let us illustrate the previous result with an example.

\begin{example}\label{exmp21}
	Let $S$ be the symmetric numerical semigroup given by
	\[S=\langle 5,7,9,11 \rangle = \{0,5,7,9,10,11,12,14,\to\}.\]
	It is clear that $\mathrm{F}(S)=13$, $\{7,12\} \subseteq \left\{x\in S \,\big\vert\, \frac{\mathrm{F}(S)}{2}<x<\mathrm{F}(S) \right\}$, and $S\setminus \{7,12\}$ is a numerical semigroup. Therefore we can apply Proposition~\ref{prop19}, getting that $\mathrm{l}(S\setminus \{7,12\})=2\times2=4$.
\end{example}

The proof of the following proposition is similar to the proof of Proposition~\ref{prop19}. So, we omit it.

\begin{proposition}\label{prop22}
	Let $S$ be a pseudo-symmetric numerical semigroup. Let us take a set $A \subseteq \left\{x\in S \,\big\vert\, \frac{\mathrm{F}(S)}{2}<x<\mathrm{F}(S) \right\}$ such that $\#(A)=n$ and $S\setminus A$ is a numerical semigroup. Then $\mathrm{l}(S\setminus A)=2n+1$. Moreover, each $(2n+1)$-semigroup can be obtained in this way.
\end{proposition}

Let us see an illustrative example of the above proposition.

\begin{example}\label{exmp24}
	Let $S$ be the pseudo-symmetric numerical semigroup given by
	\[S=\langle 5,8,11,12 \rangle = \{0,5,8,10,11,12,13,15,\to\}.\]
	Clearly $\mathrm{F}(S)=14$, $\{11,12\} \subseteq \left\{x\in S \,\big\vert\, \frac{\mathrm{F}(S)}{2}<x<\mathrm{F}(S) \right\}$, and $S\setminus \{11,12\}$ is a numerical semigroup. Therefore, by applying Proposition~\ref{prop22}, we have that $\mathrm{l}(S\setminus \{11,12\})=2\times2+1=5$.
\end{example}

Let $I$ be an irreducible numerical semigroup. Now we are interested in showing an algorithm which allows us to compute all the numerical semigroups of the form $I\setminus A$ with $A \subseteq \left\{x\in I \,\big\vert\, \frac{\mathrm{F}(I)}{2}<x<\mathrm{F}(I) \right\}$.

If $S$ is a numerical semigroup, then we denote by $\Delta(S)=\left\{s\in S \,\big\vert\, s<\frac{\mathrm{F}(S)}{2} \right\}$. The following result has an immediate proof.

\begin{proposition}\label{prop25}
	Let $I,S$ be numerical semigroups such that $I$ is irreducible, $S\subseteq I$, and $\mathrm{F}(S)=\mathrm{F}(I)$. Then $S=I\setminus A$ for some $A \subseteq \left\{\!x\in I \,\big\vert \frac{\mathrm{F}(I)}{2}<x<\mathrm{F}(I)\! \right\}$ if and only if $\Delta(S)=\Delta(I)$.
\end{proposition}

Let $A,B$ be two numerical semigroups such that $A\subseteq B$. We denote by
\[ [A,B]=\{X\mid X \mbox{ is a numerical semigroup and } A\subseteq X \subseteq B\}. \]
On the other hand, if $S$ is a numerical semigroup, then we denote by $\theta(S)$ the numerical semigroup $\langle \Delta(S) \rangle \cup \{\mathrm{F}(S)+1,\to\}$. The next result is an immediate consequence of Proposition~\ref{prop25}.

\begin{corollary}\label{cor26}
	Let $I$ be an irreducible numerical semigroup. Then $[\theta(I),I]=\left\{ I\setminus A \,\big\vert\, A\subseteq \big\{\!x\in I \,\big\vert\, \frac{\mathrm{F}(I)}{2}<x<\mathrm{F}(I)\! \big\} \mbox{ and } \, I\setminus A \, \mbox{ is a numerical semigroup} \right\}$.
\end{corollary}

At this point, our purpose is to show an algorithm to compute $[\theta(I),I]$. For that we need the concept of (rooted) tree.

A \emph{graph} $G$ is a pair $(V,E)$ where $V$ is a non-empty set (whose elements are called \emph{vertices} of $G$) and $E$ is a subset of $\{(v,w) \in V \times V \mid v \neq w\}$ (whose elements are called \emph{edges} of $G$).

A \emph{path (of length $n$) connecting the vertices $x$ and $y$ of $G$} is a sequence of different edges $(v_0,v_1),(v_1,v_2),\ldots,(v_{n-1},v_n)$ such that $v_0=x$ and $v_n=y$.

We say that a graph $G$ is a \emph{(directed rooted) tree} if there exists a vertex $r$ (known as the \emph{root} of $G$) such that, for any other vertex $x$ of $G$, there exists a unique path connecting $x$ and $r$. If $(x,y)$ is an edge of the tree, then we say that $x$ is a \emph{child} of $y$ (see \cite{rosen}).

If $A,B$ are numerical semigroups such that $A \subsetneq B$, then we denote by $\mathrm{F}_B(A)=\max(B\setminus A)$. Moreover, according to the Frobenius number $\mathrm{F}({\mathbb N})=-1$, we define $\mathrm{F}_B(B)=-1$. The following result is Lemma~4.35 of \cite{springer}.

\begin{lemma}\label{lem27}
	Let $A,B$ be two numerical semigroups such that $A \subsetneq B$. Then $A\cup \{\mathrm{F}_B(A)\}$ is another numerical semigroup. 
\end{lemma}

We define the graph $\mathrm{G}\big([A,B]\big)$ in the following way: $[A,B]$ is the set of vertices and $(X,Y)\in [A,B] \times [A,B]$ is an edge if $X\cup \{\mathrm{F}_B(X)\}=Y$. The next result is easy to prove (see Corollary~4.5 of \cite{R-variedades}).

\begin{proposition}\label{prop28}
	$\mathrm{G}\big([A,B]\big)$ is a tree with root $B$. Moreover, the children set of a vertex $P$ is $\big\{P \setminus \{x\} \in [A,B] \mid x \in \mathrm{msg}(P) \mbox{ and } x>\mathrm{F}_{B}(P) \big\}.$
\end{proposition}

Let us note that, if $I$ is an irreducible numerical semigroup and $P\in [\theta(I),I]$ with $P\not= I$, then $\mathrm{F}_I(P)=\mathrm{h}(P)$. Moreover, in order to keep the equality true and according to the value $\mathrm{F}_I(I)=-1$, we define $\mathrm{h}(I)=-1$ for any irreducible numerical semigroup $I$. Thereby, we can formulate Proposition~\ref{prop28} as follows.

\begin{corollary}\label{cor29}
	Let $I$ be an irreducible numerical semigroup. Then $\mathrm{G}\big([\theta(I),I]\big)$ is a tree with root $I$.  Moreover, if $P$ is a vertex of such a tree, then its children set is $\left\{P \setminus \{x\} \,\big\vert\, x \in \mathrm{msg}(P), \ \frac{\mathrm{F}(I)}{2}<x<\mathrm{F}(I) \mbox{ and } \, \mathrm{h}(P)<x \right\}.$
\end{corollary}

We are now ready to describe an algorithm which allows us to compute $[\theta(I),I]$. If $(x,y)$ is an ordered pair, then we denote by $\pi_1(x,y)=x$.

\begin{algorithm2}\label{alg30}
	\mbox{ }
	\begin{itemize}
		\item[] INPUT: An irreducible numerical semigroup $I$.
		\item[] OUTPUT: $[\theta(I),I]$.
		\begin{itemize}
			\item[(1)] $n=0$, $A_0=\left\{(I,-1)\right\}$, $B_0=\left\{ I \right\}$.
			\item[(2)] $n:=n+1$.
			\item[(3)] $A_n=\left\{(P\setminus\{x\},x) \,\big\vert\, (P,y)\in A_{n-1}, \,\! x\in \mathrm{msg}(P), \,\! \frac{\mathrm{F}(I)}{2}<x<\mathrm{F}(I), \,\! y<x \right\}$.
			\item[(4)] If $A_n=\emptyset$, then return $B_0\cup B_1 \cup \cdots \cup B_{n-1}$.
			\item[(5)] $B_n=\pi_1(A_n)$ and go to (2).
		\end{itemize}
	\end{itemize}
\end{algorithm2}

Let us illustrate the algorithm with an example.

\begin{example}\label{exmp31}
	It is clear that $I=\langle 5,7,9,11 \rangle$ is a symmetric numerical semigroup with $\mathrm{F}(I)=13$. Therefore, $\theta(I) = \langle 5 \rangle \cup \{14,\to\} = \langle 5,14,16,17,18 \rangle$. Let us compute $[\theta(I),I]$ with the help of Algorithm~\ref{alg30}.
	\begin{itemize}
		\item $n=0$, $A_0=\big\{(\langle 5,7,9,11 \rangle,-1)\big\}$, $B_0=\big\{\langle 5,7,9,11 \rangle\big\}$.
		\item $n=1$.
		\item $A_1=\big\{(\langle 5,9,11,12 \rangle,7), (\langle 5,7,11 \rangle,9), (\langle 5,7,9 \rangle,11) \big\}$.
		\item $B_1=\big\{\langle 5,9,11,12 \rangle, \langle 5,7,11 \rangle, \langle 5,7,9 \rangle \big\}$.
		\item $n=2$. 
		\item $A_2=\big\{(\langle 5,11,12,14,18 \rangle,9), (\langle 5,9,12,16 \rangle,11), (\langle 5,9,11,17 \rangle,12), \\ \mbox{ }\hspace{10mm} (\langle 5,7,16,18 \rangle,11) \big\}$.
		\item $B_2=\big\{\langle 5,11,12,14,18 \rangle, \langle 5,9,12,16 \rangle, \langle 5,9,11,17 \rangle, \langle 5,7,16,18 \rangle \big\}$.
		\item $n=3$.
		\item $A_3=\big\{(\langle 5,12,14,16,18 \rangle,11), (\langle 5,11,14,17,18 \rangle,12), (\langle 5,9,16,17 \rangle,12) \big\}$.
		\item $B_3=\big\{\langle 5,12,14,16,18 \rangle, \langle 5,11,14,17,18 \rangle, \langle 5,9,16,17 \rangle \big\}$.
		\item $n=4$.
		\item $A_4=\big\{(\langle 5,14,16,17,18 \rangle,12) \big\}$.
		\item $B_4=\big\{\langle 5,14,16,17,18 \rangle \big\}$.
		\item $n=5$.
		\item $A_5=\emptyset$.
	\end{itemize}
	Therefore, in this example, $[\theta(I),I]=B_0\cup B_1\cup B_2\cup B_3\cup B_4$. We depict in Figure~\ref{fig1} the tree $\mathrm{G}\big([\theta(I),I]\big)$ corresponding to this example. By the way, the number which appears over each edge $(Q,P)$ is the minimal generator $x$ of $P$ such that $Q=P\setminus\{x\}$. Note that $x=\mathrm{h}(Q)$.
	
	\begin{figure}[ht]
		\centering
		\begin{picture}(303,118)
			
			\put(194,112){$\langle 5,7,9,11 \rangle$}
			
			\put(183,95){\scriptsize 7} \put(220,98){\scriptsize 9} \put(257,101){\scriptsize 11}
			\put(160,94){\vector(3,1){42}} \put(218,94){\vector(0,1){14}} \put(274,94){\vector(-3,1){42}}
			\put(124,84){$\langle 5,9,11,12 \rangle$} \put(200,84){$\langle 5,7,11 \rangle$} \put(261,84){$\langle 5,7,9 \rangle$}
			
			\put(122,68){\scriptsize 9} \put(152,70){\scriptsize 11} \put(181,72){\scriptsize 12} \put(245,72){\scriptsize 11}
			\put(105,66){\vector(2,1){28}} \put(150,66){\vector(0,1){14}} \put(191,66){\vector(-2,1){28}} \put(255,66){\vector(-2,1){28}}
			\put(46,56){$\langle 5,11,12,14,18 \rangle$} \put(124,56){$\langle 5,9,12,16 \rangle$} \put(183,56){$\langle 5,9,11,17 \rangle$} \put(247,56){$\langle 5,7,16,18 \rangle$}
			
			\put(66,40){\scriptsize 11} \put(103,45){\scriptsize 12} \put(173,45){\scriptsize 12}
			\put(55,38){\vector(4,3){19}} \put(110,38){\vector(-4,3){19}} \put(180,38){\vector(-4,3){19}}
			\put(7,28){$\langle 5,12,14,16,18 \rangle$} \put(85,28){$\langle 5,11,14,17,18 \rangle$} \put(170,28){$\langle 5,9,16,17 \rangle$}
			
			\put(44,14){\scriptsize 12}	
			\put(42,10){\vector(0,1){14}}
			\put(7,0){$\langle 5,14,16,17,18 \rangle$}
		\end{picture}
		\caption{Tree $\mathrm{G}\big([\theta(I),I]\big)$ with $I=\langle 5,7,9,11 \rangle$.}
		\label{fig1}
	\end{figure}
\end{example}

Let $G$ be a tree with root $r$. The \emph{depth} of a vertex $x$ of $G$ is the number of edges in the unique path connecting $x$ and $r$. The \emph{$n$-level} of $G$ is the set of all vertices with depth equal to $n$. Lastly, the \emph{height} of $G$ is the maximum depth of any vertex of $G$. (See \cite{rosen}.)

Let us observe that, in Example~\ref{exmp31}, we have a tree of height $4$. Moreover, the $n$-level of such a tree is $B_n$ (for $n=0,1,2,3,4$).

The following result is easy to prove.

\begin{proposition}\label{prop32}
	Let $I$ be an irreducible numerical semigroup.
	\begin{enumerate}
		\item The height of the tree $\mathrm{G}\big([\theta(I),I]\big)$ is equal to $\#(I\setminus\theta(I))$.
		\item If $S$ belongs to the $n$-level of $\mathrm{G}\big([\theta(I),I]\big)$, then $\mathrm{l}(S)=2n+\mathrm{l}(I)$.
		\item $\left\{\mathrm{l}(S) \mid S\in[\theta(I),I] \right\} = \left\{2n+\mathrm{l}(I) \mid n\in \{0,1,\ldots,\#(I\setminus\theta(I) \} \right\}$.
	\end{enumerate}
\end{proposition}

Let us illustrate the content of this proposition with an example.

\begin{example}\label{exmp33}
	Following Example~\ref{exmp31}, we know that $I=\langle 5,7,9,11 \rangle$ is a symmetric numerical semigroup. Therefore, $I$ is an irreducible numerical semigroup with $\mathrm{l}(I)=0$ (see Proposition~\ref{prop3}).
	It is clear that $\theta(I) = \langle 5 \rangle \cup \{14,\to\} = \langle 5,14,16,17,18 \rangle$ and $I\setminus \theta(I)=\{7,9,11,12\}$. Thus, $\mathrm{G}\big([\theta(I),I]\big)$ is a tree with height equal to $4$. Consequently, $\left\{\mathrm{l}(S) \mid S\in[\theta(I),I] \right\}=\{0,2,4,6,8\}$. Moreover, $\big\{S\in [\theta(I),I] \mid \mathrm{l}(S)=6 \big\} = \big\{ S\in \mathrm{G}\big([\theta(I),I]\big) \mid S \mbox{ has depth equal to } 3\big\} = B_3$. 
\end{example}

Let us observe that, if $n\in{\mathbb N}$, we can compute all the $2n$-semigroups ($(2n+1)$-semigroups, respectively) in two steps.
\begin{enumerate}
	\item Compute the set $\mathcal{F}(n)$ ($\mathcal{G}(n)$, respectively) formed by all symmetric numerical semigroups (pseudo-symmetric numerical semigroups, respectively) $I$ such that $\#(I\setminus\theta(I))\geq n$.
	\item For each $I\in \mathcal{F}(n)$ ($\mathcal{G}(n)$, respectively) compute the $n$-level of $\mathrm{G}\big([\theta(I),I]\big)$.
\end{enumerate}

Note that there exist infinitely many irreducible numerical semigroups $I$ such that $\#(I\setminus\theta(I))=n$. Therefore, there exist infinitely many $2n$-semigroups and infinitely many $(2n+1)$-semigroups. In the next section, we provide an algorithm which allows us to build the finite set consisting of all $k$-semigroups which have a given prescribed Frobenius number $F$.

\section{${\boldsymbol k}$-semigroups with a fixed Frobenius number}\label{sect-k-ns}

In \cite{forum} it is shown a rather efficient algorithmic process to compute all the irreducible numerical semigroups with a prescribed Frobenius number. Let us briefly describe the basic idea of the algorithm.

Let $\mathcal{I}(F)$ the set of all irreducible numerical semigroup with Frobenius number equal to $F$. Then $\mathrm{G}(\mathcal{I}(F))$ is the graph where $\mathcal{I}(F)$ is the set of vertices and $(T,S)\in \mathcal{I}(F) \times \mathcal{I}(F)$ is an edge if $\mathrm{m}(T)<\frac{F}{2}$ and $S=(T\setminus\{\mathrm{m}(T) \}) \cup \left\{ F-\mathrm{m}(T) \right\}$.

The next results are Proposition~2.7 and Theorem~2.9 of \cite{forum}, respectively.

\begin{proposition}\label{prop34}
	Let $F$	be a positive integer. Then there exists a unique irreducible numerical semigroup $\mathrm{C}(F)$ with Frobenius number $F$ and multiplicity greater than $\frac{F}{2}$. Moreover,
	\[\mathrm{C}(F) = \left\{ \begin{array}{ll}
		\vspace{1pt} \left\{0,\frac{F+1}{2},\to \right\} \setminus \{F\}, & \mbox{if $F$ is odd,} \\[3pt] \left\{0,\frac{F}{2}+1,\to \right\} \setminus \{F\}, & \mbox{if $F$ is even.}
	\end{array} \right.\]
\end{proposition}

\begin{proposition}\label{prop35}
	Let $F$	be a positive integer. Then $\mathrm{G}(\mathcal{I}(F))$ is a tree with root $\mathrm{C}(F)$. Moreover, if $S\in \mathcal{I}(F)$, then the children set of $S$ is
	$\Big\{(S\setminus \{x\}) \cup \{F-x\} \,\big\vert$ $ x\in\mathrm{msg}(S), \ \frac{F}{2}<x<F, \ 2x-F\notin S, \ 3x\not=2F, \ 4x\not=3F \mbox{ and } \, F-x<\mathrm{m}(S) \Big\}.$
\end{proposition}

Let us illustrate both of the above results with an example.

\begin{example}\label{exmp36}
	We want to compute all the irreducible numerical semigroups with Frobenius number $11$. For that, we build the tree $\mathrm{G}(\mathcal{I}(11))$ starting at its root $\mathrm{C}(11)=\langle 6,7,8,9,10 \rangle$ and adding to the known vertices their children (that are given by Proposition~\ref{prop35}). Thus we have the tree of Figure~\ref{fig2} (for more details see Example~2.10 of \cite{forum}).
	
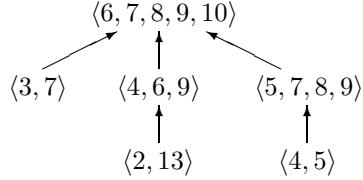
	\begin{figure}[ht]
		\centering
		\begin{picture}(146,62)
			
			\put(38,56){$\langle 6,7,8,9,10 \rangle$}
			
			\put(20,38){\vector(2,1){28}} \put(64,38){\vector(0,1){14}} \put(110,38){\vector(-2,1){28}}
			\put(7,28){$\langle 3,7 \rangle$} \put(48,28){$\langle 4,6,9 \rangle$} \put(100,28){$\langle 5,7,8,9 \rangle$}
			
			\put(64,10){\vector(0,1){14}} \put(120,10){\vector(0,1){14}}
			\put(50,0){$\langle 2,13 \rangle$} \put(109,0){$\langle 4,5 \rangle$}
		\end{picture}
		\caption{Tree of irreducible numerical semigroups with Frobenius number $11$.}
		\label{fig2}
	\end{figure}
\end{example}

The next result gives us the conditions that must be satisfied by two non-negative integers $k$ and $F$ in order to have at least one $k$-semigroup with Frobenius number $F$. As usual, if $q$ is a rational number, then we denote by $\lfloor q \rfloor = \max \{z \in {\mathbb Z} \mid z\leq q \}$.

\begin{proposition}\label{prop37}
	Let $k,F$ be non-negative integers. Then there exists at least one $k$-semigroup with Frobenius number $F$ if and only if $k+F$ is an odd number and $F\geq k+1$.
\end{proposition}

\begin{proof}
	\textit{(Necessity.)} This is Corollary~\ref{cor_odd}.
	
	\textit{(Sufficiency.)} It is enough to take $S=\mathrm{C}(F) \setminus A$ where $A$ is a $\left\lfloor \frac{k}{2} \right\rfloor$-cardinality subset of $\{x\in \mathrm{C}(F) \mid 0\not=x<F \}$ and apply Propositions~\ref{prop19} or \ref{prop22} depending on whether $F$ is odd or even.
\end{proof}

Let us see an example illustrating the previous proposition.

\begin{example}\label{exmp38}
	Let us build an $8$-semigroup with Frobenius number $17$. Since $8+17$ is odd and $17\geq 8+1$, by applying Proposition~\ref{prop37}, we know that there exists at least one. Moreover, from the proof of such a proposition, we have that $S=\mathrm{C}(17) \setminus \{9,11,13,16\} = \{0,10,12,14,15,18,\to \}$ is an $8$-semigroup with Frobenius number $17$.
\end{example}

We are now interested in determining all the irreducible numerical semigroups $I$ such that the set $[\theta(I),I]$ contains at least one $k$-semigroup. The following result can be easily deduced from Proposition~\ref{prop32}.

\begin{proposition}\label{prop39}
	Let $I$ be an irreducible numerical semigroup and $k\in{\mathbb N}$ such that $\mathrm{F}(I)+k$ is an odd number and $\mathrm{F}(I)\geq k+1$. Then $[\theta(I),I]$ contains at least one $k$-semigroup if and only if $\#(I\setminus \theta(I)) \geq \left\lfloor \frac{k}{2} \right\rfloor$.
\end{proposition}

We are ready to show an algorithm which allows us to compute all the $k$-semigroups with a fixed Frobenius number.

\begin{algorithm2}\label{alg40}
	\mbox{ }
	\begin{itemize}
		\item[] INPUT: Two non-negative integers $k$ and $F$.
		\item[] OUTPUT: The set of all $k$-semigroups with Frobenius number $F$.
		\begin{itemize}
			\item[(1)] If $F+k$ is even, then return $\emptyset$.
			\item[(2)] If $F<k+1$, then return $\emptyset$.
			\item[(3)] Compute the set $\mathcal{F} = \left\{ I\in \mathcal{I}(F) \,\big\vert\, \#(I\setminus \theta(I)) \geq \left\lfloor \frac{k}{2} \right\rfloor \right\}$.
			\item[(4)] For each $I\in \mathcal{F}$, compute the set
			\[\mathrm{D}(I)=\left\{S\in \mathrm{G}\big([\theta(I),I]\big) \,\big\vert\, S \mbox{ belongs to the $\left\lfloor \frac{k}{2} \right\rfloor$-level} \right\}.\]
			\item[(5)] Return $\bigcup_{I\in \mathcal{F}} \mathrm{D}(I)$. 
		\end{itemize}
	\end{itemize}
\end{algorithm2}

\begin{remark}\label{rem40b}
	It is clear that, if $I_1,I_2$ are irreducible numerical semigroups (with the same Frobenius number), then $\Delta(I_1)=\Delta(I_2)$ if and only if $I_1=I_2$. Consequently, $[\theta(I_1),I_1] \cap [\theta(I_2),I_2] \not= \emptyset$ if and only if $I_1=I_2$. Therefore, $\bigcup_{I\in \mathcal{F}} \mathrm{D}(I)$ is a partition of the set of all $k$-semigroups with Frobenius number $F$ (see \cite{computer}).
\end{remark}

The next result is useful to carry out the step~(3) in the previous algorithm.

\begin{proposition}\label{prop41}
	Let $F$ be a positive integer.
	\begin{enumerate}
		\item $\#\big(\mathrm{C}(F)\setminus \theta(\mathrm{C}(F))\big) = \left\lfloor \frac{F-1}{2} \right\rfloor$.
		\item If $Q$ is a child of $P$ in the tree $\mathrm{G}(\mathcal{I}(F))$, then $\#(Q\setminus \theta(Q)) < \#(P\setminus \theta(P))$.
	\end{enumerate}
\end{proposition}

\begin{proof}
	The first statement is trivial. For the second one, having in mind that $Q=(P\setminus \{x\}) \cup \{F-x\}$ for some $x\in\mathrm{msg}(P)$ such that $\frac{F}{2}<x<F$, then we easily deduce that $Q\setminus \theta(Q) \subsetneq P\setminus \theta(P)$ and, therefore, $\#(Q\setminus \theta(Q)) < \#(P\setminus \theta(P))$.
\end{proof}

Let us see an illustrative example of Algorithm~\ref{alg40}.

\begin{example}\label{exmp42}
	Let us suppose that we want to build all the $6$-semigroups with Frobenius number $11$. For that we use Algorithm~\ref{alg40}.
	\mbox{ }
	\begin{itemize}
		\item[(1)] $11+6$ is odd.
		\item[(2)] $11\geq 6+1$.
		\item[(3)] By applying the algorithm of \cite{forum} which we commented at the beginning of this section and having in mind Proposition~\ref{prop41}, we get that
		\[\mathcal{F} = \big\{I \in \mathcal{I}(11) \mid \#(I\setminus \theta(I)) \geq 3 \big\} = \big\{ \langle 6,7,8,9,10 \rangle, \langle 4,6,9 \rangle, \langle 5,7,8,9 \rangle \big\}.\]
		Let us observe that it is not necessary to compute all the numerical semigroups of $\mathcal{I}(11)$. In fact, from Example~\ref{exmp36} and Proposition~\ref{prop41}, we just need to take into account $\mathrm{C}(11)=\langle 6,7,8,9,10 \rangle$, $I_1=\langle 3,7 \rangle$, $I_2=\langle 4,6,9 \rangle$, and $I_3=\langle 5,7,8,9 \rangle$ because $\#\big(\mathrm{C}(11)\setminus \theta(\mathrm{C}(11))\big)=5$, $\#(I_1\setminus \theta(I_1))=2$, and $\#(I_2\setminus \theta(I_2))=\#(I_3\setminus \theta(I_3))=3$.
		\item[(4)] By using Algorithm~\ref{alg30}, we get that
		\begin{itemize}
			\item[$\bullet$] $\mathrm{D}(\mathrm{C}(11)) = \mathrm{D}(\langle 6,7,8,9,10 \rangle) = \big\{ \langle 6,7,15,16,17 \rangle, \langle 6,8,13,15,17 \rangle, $ \\[1pt]
			\mbox{ } \hspace{10pt} $\langle 6,9,13,14,16,17 \rangle, \langle 6,10,13,14,15,17 \rangle, \langle 7,8,12,13,17,18 \rangle, $ \\[1pt]
			\mbox{ } \hspace{10pt} $\langle 7,9,12,13,15,17 \rangle, \langle 7,10,12,13,15,16,18 \rangle, $ \\[1pt]
			\mbox{ } \hspace{10pt} $\langle 8,9,12,13,14,15,19 \rangle, \langle 8,10,12,13,14,15,17,19 \rangle, $ \\[1pt]
			\mbox{ } \hspace{10pt} $\langle 9,10,12,13,14,15,16,17 \rangle \big\}$;
			\item[$\bullet$] $\mathrm{D}(I_2) = \mathrm{D}(\langle 4,6,9 \rangle) = \big\{ \langle 4,13,14,15 \rangle \big\}$;
			\item[$\bullet$] $\mathrm{D}(I_3) = \mathrm{D}(\langle 5,7,8,9 \rangle) = \big\{ \langle 5,12,13,14,16 \rangle \big\}$. 
		\end{itemize}
		\item[(5)] $\mathrm{D}(\mathrm{C}(11)) \cup \mathrm{D}(I_2) \cup \mathrm{D}(I_3)$ is the set of all $6$-semigroups with Frobenius number $11$.
	\end{itemize}
\end{example}

\section{Wilf's conjecture and ${\boldsymbol k}$-semigroups}\label{sect-wilf}

From Lemma~\ref{lem2} we have that $\mathrm{g}(S)=\frac{\mathrm{F}(S)+1+\mathrm{l}(S)}{2}$. Moreover, from the definitions, $\mathrm{c}(S)=\mathrm{F}(S)+1$. Given these two facts we can rewrite Wilf's conjecture in terms of the Frobenius number and the number of second kind gaps.
\begin{lemma}
	Let $S$ be a numerical semigroup. Then $S$ fulfils Wilf's conjecture if and only if $\mathrm{l}(S) \leq \frac{\mathrm{e}(S)-2}{\mathrm{e}(S)}(\mathrm{F}(S)+1)$.
\end{lemma}

\begin{proof}
	Let us recall that $S$ fulfils Wilf's conjecture if $\frac{\mathrm{g}(S)}{\mathrm{c}(S)}\leq1-\frac{1}{\mathrm{e}(S)}$. Since
	\[ \frac{\mathrm{g}(S)}{\mathrm{c}(S)}\leq1-\frac{1}{\mathrm{e}(S)} \Leftrightarrow \mathrm{g}(S) \leq \frac{\left(\mathrm{e}(S)-1\right)}{\mathrm{e}(S)} \, \mathrm{c}(S) \Leftrightarrow \]
	\[ \frac{\mathrm{F}(S)+1+\mathrm{l}(S)}{2} \leq  \frac{\left(\mathrm{e}(S)-1\right)}{\mathrm{e}(S)} \left( \mathrm{F}(S)+1 \right) \Leftrightarrow \]
	\[ \mathrm{l}(S) \leq \left( 2\,\frac{\left(\mathrm{e}(S)-1\right)}{\mathrm{e}(S)} - 1 \right) \left( \mathrm{F}(S)+1 \right), \]
	we have the conclusion.
\end{proof}

As an application of the previous lemma, let us see some remarks on $k$-semigroups, with $k\in\{0,1,2,3\}$, and Wilf's conjecture. In the first two we recover well known results.

\begin{remark}
	If $S=\mathbb{N}$, then $\mathrm{l}(S)=0$, $\mathrm{e}(S)=1$, $\mathrm{F}(S)=-1$, and it is clear that $0 = \frac{1-2}{1}(-1+1)$. On the other hand, it is trivial that $0 \leq \frac{e-2}{e}(F+1)$ for all $e\geq 2$ and $F\geq 1$. Therefore, we have that all 0-semigroups fulfil Wilf's conjecture.
\end{remark}

\begin{remark}\label{rem-e3}
	Let $S_1$ be a numerical semigroup such that $\mathrm{e}(S_1)=2$. Then $S_1$ is symmetric and, consequently, $S_1$ is a 0-semigroup. Thus, if $S_2$ is a 1-semigroup, we have that $\mathrm{e}(S_2)\geq3$. Moreover, by Corollary~\ref{cor_odd}, $\mathrm{F}(S_2)$ is an even number greater than or equal to 2. Now, because $1\leq\frac{e-2}{e}(2+1)$ for all $e\geq 3$, we deduce that all 1-semigroups fulfil Wilf's conjecture.
\end{remark}

\begin{remark}\label{rem-l2}
	From Remark~\ref{rem-e3} we know that, if $S$ is a 2-semigroup, then $\mathrm{e}(S)\geq3$. Moreover, by Corollary~\ref{cor_odd}, $\mathrm{F}(S)$ is an odd number greater than or equal to 3. Since $S=\langle 4,5,6,7 \rangle$ is the unique 2-semigroup with $\mathrm{F}(S)=3$ and $2\leq\frac{e-2}{e}(5+1)$ for all $e\geq 3$, we get that all 2-semigroups fulfil Wilf's conjecture.
\end{remark}

\begin{remark}\label{rem-l3}
	From Remark~\ref{rem-e3} we know that, if $S$ is a 3-semigroup, then $\mathrm{e}(S)\geq3$. Moreover, by Corollary~\ref{cor_odd}, $\mathrm{F}(S)$ is an even number greater than or equal to 4. Now, it is clear that
	\begin{itemize}
		\item $S=\langle 5,6,7,8,9 \rangle$ is the unique 3-semigroup with $\mathrm{F}(S)=4$;
		\item $S=\langle 4,7,9,10 \rangle$ and $S=\langle 5,7,8,9,11 \rangle$ are the unique 3-semigroups with $\mathrm{F}(S)=6$;
		\item $3\leq\frac{e-2}{e}(8+1)$ for all $e\geq 3$.
	\end{itemize}
	Thus, we conclude that all 3-semigroups fulfil Wilf's conjecture.
\end{remark}

We finish this section with a remark on $k$-semigroups, with $k\geq 4$, and Wilf's conjecture.

\begin{remark}\label{rem-final}
	From Remark~\ref{rem-e3} we have that, if $S$ is a $k$-semigroup with $k\geq 4$, then $\mathrm{e}(S)\geq3$. Moreover, it is well known that all numerical semigroups with embedding dimension less than or equal to three satisfy Wilf's conjecture (see \cite[Remark after Theorem 20]{froberg} or \cite[Corollary 2.6]{dobbs-matthews}). Thus, we can focus our attention on the case $\mathrm{e}(S)\geq4$. Now then, since $k\leq\frac{e-2}{e}((2k-1)+1)$ for all $e\geq 4$, we conclude that, if $k,F\in\mathbb{N}$, $k\geq 4$, and $F\geq2k-1$, then all $k$-semigroups with Frobenius number $F$ verify Wilf's conjecture. On the other hand, it is clear that $S=\langle k+2, k+3, \ldots, 2k+3 \rangle$ is the unique $k$-semigroup with Frobenius number equal to $k+1$ and, moreover, $S$ fulfils Wilf's conjecture (in fact, $S$ achieves the equality). Finally, by Corollary~\ref{cor_odd}, the Frobenius number $F$ of any $k$-semigroup satisfies the inequality $F\geq k+1$ and, in addition, $F+k$ has to be an odd number. Thereby, the following question arises in order to solve the Wilf's conjecture: if $S$ is a $k$-semigroup with $\mathrm{e}(S)\geq 4$, $\mathrm{k}\geq 4$, and Frobenius number in the interval $[k+3,2k-1)$, does $S$ fulfil Wilf's conjecture? (in fact, because $[k+3,2k-1) = \emptyset$ for $k=4$, we can take $k\geq5$).
\end{remark}

\section*{Acknowledgement}

This version of the article has been accepted for publication, after peer review but is not the Version of Record and does not reflect post-acceptance improvements, or any corrections. The Version of Record is available online at: http://dx.doi.org/10.1007/s00025-021-01542-y.

\end{document}